\documentclass[a4paper, 11pt]{article}

\usepackage{a4wide}
\usepackage{amsmath}
\usepackage{amssymb}
\usepackage{amsthm}
\usepackage{graphicx}%

\DeclareMathOperator{\dist}{\textit{d}}
\DeclareMathOperator{\diam}{diam}
\DeclareMathOperator{\edges}{\textit{e}}
\DeclareMathOperator{\er}{\textit{e}^{\textit{r}} }

\newtheorem{theorem}{Theorem}

\newtheorem{proposition}[theorem]{Proposition}

\begin{document}

\title{Growth of graph powers}
\author{A. Pokrovskiy (LSE)}
\maketitle

\begin{abstract}
For a graph $G$, its $r$th power is constructed by placing an edge between two vertices if they are within distance $r$ of each other.  In this note we study the amount of edges added to a graph by taking its $r$th power.
In particular we obtain that either the $r$th power is complete or ``many" new edges are added.  This is an extension of a result obtained by P.~Hegarty for cubes of graphs.
\end{abstract}

\section{Introduction}
This note addresses some questions raised by P. Hegarty in \cite{Hegarty}.  In that paper he studied results about graphs inspired by the Cauchy-Davenport Theorem.

All graphs in this paper are simple and loopless.
For two vertices $u, v \in V(G)$, denote the length of the shortest path between them by $\dist(u,v)$.  
For $v \in V(G)$, define its \emph{ith neighborhood} as $N_i(v) = \{ u \in V(G): \dist(u,v)=i\}$.
The $r$th power of a graph $G$, denoted $G^r$, is constructed from $G$ by adding an edge between two vertices $x$ and $y$ when they are within distance $r$ in $G$.
Define the diameter of $G$, $\diam(G)$, as the minimal $r$ such that $G^r$ is complete (alternatively, the maximal distance between two vertices).  
Denote the number of edges of $G$ by $e(G)$.      
For $v \in V(G)$ and a set of vertices $S$, define $\er (v,S)= |\{u\in S: \dist(v,u)\leq r \}|$.

  The Cayley graph of a subset $A \subseteq \mathbb{Z}_p$ is constructed on the vertex set $\mathbb{Z}_p$.  For two distinct vertices $x, y \in \mathbb{Z}_p$, we define $xy$ to be an edge whenever $x - y \in A$ or $y - x \in A$.
The following is a consequence of the Cauchy-Davenport Theorem (usually stated in the language of additive number theory \cite{cdt}).
\begin{theorem}\label{thm:cauchydavenport}
Let $p$ be a prime, $A$ a subset of $\mathbb{Z}_p$, and $G$ the Cayley graph of $A$.  Then for any integer $r < \mathrm{diam}(G)$:
$$\edges(G^r) \geq r \edges(G) .$$
\end{theorem}
If we take $A$ to be the arithmetic progression $\{a, 2a,\dots , ka\}$, then equality holds in this theorem for all $r < \diam(G)$.
We might look for analogues of Theorem \ref{thm:cauchydavenport} for more general graphs $G$.  In particular since these Cayley graphs are always regular and (when $p$ is prime) connected, we might focus on regular, connected $G$.
In \cite{Hegarty} Hegarty proved the following theorem:
\begin{theorem}\label{thm:hegarty}
Suppose $G$ is a regular, connected graph with $\diam(G)\geq 3$.  Then we have
$$\edges(G^3) \geq (1+\epsilon) \edges(G) ,$$
with $\epsilon \approx 0.087$
\end{theorem}
In other words, the cube of $G$ retains the original edges of $G$ and gains a positive proportion of new ones.  In Section 3 we prove this theorem with an improved constant of $\epsilon = \frac{1}{6}$. The requirement of regularity cannot be easily dropped, as shown in \cite{Hegarty}.

Theorem \ref{thm:hegarty} leads to the question of how the growth behaves for other powers of the $G$.  Note that Theorem \ref{thm:hegarty} cannot be used recursively to obtain such a result -- since the cube of a regular graph is not necessarily regular.  In \cite{Hegarty} it was shown that no equivalent of Theorem \ref{thm:hegarty} exists with $G^3$ replaced by $G^2$, and it was asked what happens for higher powers.  In this note we address that question.

\section{Main Result}
We prove the following theorem:

\begin{theorem}\label{thm:higherpowers}
Suppose $G$ is a regular, connected graph, and $r \leq \diam(G)$.
Then we have:
\[
\edges(G^{r}) \geq \left( \left\lceil \frac{r}{3} \right\rceil -1 \right)\edges(G) .
\]

\end{theorem}
\begin{proof}
Let the degree of each vertex be $d$.    Fix some $v$ with $N_{\mathrm{diam}(G)}(v)$ nonempty.  

Consider any vertex $u\in V(G)$.  Then for any $j$
satisfying $d(u,v)-r<j \leq d(u,v)$, there is a $w_j  \in N_j(v)$ such that
$\dist(u, w_j)<r$.  For such a $w_j$, all vertices $x \in N_1(w_j)$
have $\dist(u,x)\leq r$. All such $x$ are contained in
$N_{j-1}(v)\cup N_j(v)\cup N_{j+1}(v)$, hence
\begin{equation} \label{eq:bound} \tag{1}
\er (u, N_{j-1}(v)\cup N_j(v)\cup N_{j+1}(v)) \geq d.
\end{equation}
 Note that each $j \in \{ d(u,v)-3, d(u,v)-6, \dots,  d(u,v)-3\left( \left\lceil \frac{1}{3}\min\{d(u,v),r\} \right\rceil -1 \right) \}$ satisfies  $d(u,v)-r<j\leq d(u,v)$.  Summing the bound (\ref{eq:bound}) over all these $j$, noting that any edge is counted at most once, we obtain
\begin{align*}
 \er (u, N_0(v)\cup \dots \cup N_{d(u,v)-2}(v)) %&=  \sum_{k=1} ^ {\left\lceil\frac{1}{3}\min(i,r) \right\rceil -1} \er (u, N_{i-3k-1}(v)\cup N_{i-3k}(v) \cup N_{i-3k+1}(v))
&\geq \left \lceil \frac{1}{3}\min\{d(u,v),r\} \right \rceil d -d  .
\end{align*}

Now we sum this over all $u\in G$.  Note that since the edges counted above go from some $N_i(v)$ to $N_j(v)$ with $j<i$, each edge is counted at most once.  Also we haven't yet counted any of the original edges of $G$, so we might as well add them.  Hence
\begin{align*}
\edges(G^r) &\geq   \sum_{u\in G} \er(u, N_0(v)\cup\dots \cup N_{\dist(u,v)-2}(v)) + \edges(G)
\\     &\geq   \sum_{u\in G} \left\lceil \frac{1}{3}\min\{\dist(u,v),r\} \right\rceil  d  - |V(G)|d + \edges(G)
\\     &=   \sum_{u\in G} \left\lceil \frac{1}{3}\min\{\dist(u,v),r\} \right\rceil  - \edges(G) . \label{eq:vsum} \tag{2}
\end{align*}

%reverted------------------------------
Obviously there was nothing particularly special about $v$.  We can get a similar expresssion using $v' \in N_{ \mathrm{diam}(G)}(v)$, namely
\begin{equation*}\label{eq:vsum2}
\edges(G^r) \geq   \sum_{u\in G} \left\lceil \frac{1}{3}\min\{\dist(u,v'),r\} \right\rceil  - \edges(G) .  \tag{3}
\end{equation*}
Averaging (\ref{eq:vsum}) and (\ref{eq:vsum2}) we get

\begin{equation*}\label{eq:vsumaverage}
\edges(G^r) \geq \frac{1}{2} \sum_{u\in G} \left( \left\lceil \frac{1}{3} \min\{\dist(u,v),r\} \right\rceil+\left\lceil \frac{1}{3} \min\{\dist(u,v'),r\}\right\rceil \right)d  - \edges(G). \tag{4}
\end{equation*}
Note that for any $u \in V(G)$ we have
\begin{equation*}\label{eq:sumbound}
 \left\lceil \frac{1}{3} \min\{\dist(u,v),r\} \right\rceil+\left\lceil \frac{1}{3} \min\{\dist(u,v'),r\}\right\rceil \geq\left\lceil \frac{r}{3}\right\rceil. \tag{5}
\end{equation*}
This is because $d(u,v)+d(u,v')\geq d(v,v')=\diam(G) \geq r$.  %In principle for all $u \notin \{v,v'\}$ we could replace $\left\lceil \frac{r}{3}\right\rceil$ in (\ref{eq:sumbound}) by $\left\lceil \frac{r+1}{3}\right\rceil$ and slightly strengthen the result. 
Putting the bound (\ref{eq:sumbound}) into the sum (\ref{eq:vsumaverage}) we obtain
\begin{equation*}
 e(G^r) \geq \frac{|V(G)|d}{2}\left\lceil \frac{r}{3}\right\rceil   -\edges(G) 
 = \left\lceil \frac{r}{3} \right\rceil \edges(G)-\edges(G) .
\end{equation*}
Thus the theorem is proven.
\end{proof}

\section{Cubes}
Note that for $r \leq 6$ the bounds in Theorem \ref{thm:higherpowers} are trivial. In particular it says nothing about the increase in the number of edges of the cube of a regular, connected graph. Such an increase was already demonstrated by Hegarty in Theorem \ref{thm:hegarty}.
Here we give an alternative proof of that theorem, yielding a slightly better constant.
\begin{theorem}\label{thm:improvedhegarty}
Suppose $G$ is a regular, connected graph with $\diam(G)\geq3$.  Then we have
$$\edges(G^3) \geq \left(1+\frac{1}{6} \right) \edges(G) .$$
\end{theorem}
\begin{proof}
Let the degree of each vertex be $d$.   Note that as $G$ is regular, and not complete, every $v \in V(G)$ will have a non-neighbour in $G$.  Together with connectedness this implies that each $v \in V(G)$ has at least one new neighbour in $G^2$.  This implies the theorem for $d \leq 6$.  For the remainder of the proof, we assume that $d > 6$. The proof rests on the following colouring of the edges of $G$:  For  an edge $uv$ in $G$, colour
\begin{align*}
uv \ \text{\bf{red}}\text{ if } &|N_1(u)\cap N_1(v)| > \frac{2}{3} d,
\\uv \ \text{\bf{blue}}\text{ if } &|N_1(u)\cap N_1(v)| \leq \frac{2}{3} d.
\end{align*}

Notice that if $uv$ is a blue edge, then there are at least $\frac{4}{3}d-1$  neighbours of $u$ in $G^2$.  This is because $u$ will be connected to everything in $N_1(u)\cup N_1(v)$ except itself, and $|N_1(u)\cup N_1(v)|\geq \frac{4}{3}d$ for $uv$ blue.
If, in addition, we have some $x$ connected to $u$ by an edge (of any colour), then $x$ will be at distance at most $3$ from everything in $N_1(u) \cup N_1(v) \setminus \{x\}$. Hence $x$ will have at least $\frac{4}{3}d-1$ neighbours in $G^3$.
%  Similarly if $x$ is connected to $u$ with $uv$ a blue edge for some $v$, then $x$ has at least $\frac{4}{3}d$  neighbours of in $G^3$. 
\\
Partition the vertices of $G$ as follows:

$B = \{v \in V(G) : v \text{ has a blue edge coming out of it}\},$
 
$R = \{v \in V(G) : v\notin B \text{ and there is a } u \in B \text{ such that } uv \text{ is an edge}  \},$

$S = V(G)\setminus (B\cup R)  .$ 
\\
By the above argument, if $v$ is in $B \cup R$, then $e^3(v,V(G))\geq \frac{4}{3}d-1$.  Recall that each $u \in S$ will have at least one new neighbour in $G^2$, giving $e^3(u,V(G))\geq d+1$.  Summing these two bounds over all vertices in $G$, noting that any edge is counted twice, gives
\begin{align*}
2e(G^3) &\geq  \left(\frac{1}{3}d-1\right)|B\cup R|+(d+1)|S|
\\ &=\left(\frac{4}{3}d-1\right)|B\cup R|+ (d+1)\left( |V(G)|-|B\cup R| \right)
\\ &= \frac{7}{6}d|V(G)|+\frac{1}{3}\left(|B\cup R|-\frac{1}{2}|V(G)| \right) \left(d-6 \right)
\\ &= \frac{7}{3}e(G) +\frac{1}{3}\left(|B\cup R|-\frac{1}{2}|V(G)| \right) \left(d-6 \right).
\end{align*}

Recall that we are considering the case when $d > 6$.  Thus to prove that $e(G^3)\geq \frac{7}{6}e(G)$, it suffices to show that $|B \cup R|\geq \frac{1}{2}|V(G)|$.  To this end we shall demonstrate that $|S|\leq |R|$.  First however we need a proposition helping us to find blue edges in $G$.
\begin{proposition} \label{cliqueproposition}
 For any $v \in V(G)$ there is some $b \in B$ such that $d(v,b) \leq 2$.
\end{proposition}
\begin{proof}
Suppose $d(v,u)=3$.  Then there are vertices $x$ and $y$ such that $\{v, x, y, u\}$ forms a path between $u$ and $v$.  We will show that one of the edges $vx$, $xy$ or $yu$ is blue.  This will prove the proposition assuming that there are \emph{any} blue edges to begin with.  However, it also shows the existence of blue edges because $\diam(G)\geq 3$.  

So, suppose that the edges $vx$ and $uy$ are red.
Then we have $|N_1(v)\cap N_1(x)| > \frac{2}{3}d$, and $|N_1(u)\cap N_1(y)| > \frac{2}{3}d$.  Using this and $|N_1(u)\cap N_1(v)| = \emptyset$ gives

\begin{align*}
|N_1(x)\cup N_1(y)| &\geq |(N_1(x)\cup N_1(y))\cap N_1(v)| +|(N_1(x)\cup N_1(y))\cap N_1(u)|
\\ &\geq |N_1(x) \cap N_1(v)| +|N_1(y)\cap N_1(u)|
\\ &> \frac{4}{3}d.
\end{align*}
Therefore $|N_1(x)\cap N_1(y)|=2d-|N_1(x)\cup N_1(y)| \leq \frac{2}{3}d$.  Hence $xy$ is blue, proving the proposition.  
\end{proof}

Now we will show that $|S|\leq |R|$.
Suppose $r \in R$.  By the definition of $R$, there is a $b\in B$ such that $rb$ is an edge.  This edge is neccesarily red as $r \notin B$. Using $N_1(b)\subseteq B\cup R$,we have $|N_1(r)\cap (B \cup R)| \geq |N_1(r)\cap N_1(b)| > \frac{2}{3}d$. Hence
\begin{equation}\label{eq:upperbound}
|N_1(r)\cap S |\leq \frac{1}{3}d. \tag{6}
\end{equation}

Suppose $s \in S$.  %The definition of $S$ implies that $N_1(s)\subseteq R\cup S$. 
Proposition \ref{cliqueproposition} implies that there is some $r \in R$ such that $sr$ is an edge. Since $sr$ is red, we have $|N_1(s)\cap N_1(r)| > \frac{2}{3}d$.  Using this, the fact that $N_1(s)\subseteq R\cup S$, and (\ref{eq:upperbound}), gives
\begin{align*}\label{eq:lowerbound}
|N_1(s)\cap R| &\geq  |N_1(s)\cap N_1(r)\cap R|
\\             &=|N_1(s) \cap N_1(r)|- |N_1(s)\cap N_1(r)\cap S| 
\\             &\geq |N_1(s) \cap N_1(r)|- |N_1(r)\cap S|
\\             &> \frac{1}{3}d. \tag{7}
\end{align*}

Double-counting the edges between $S$ and $R$ using the bounds (\ref{eq:upperbound}) and (\ref{eq:lowerbound}) gives a contradiction unless $|S| \leq |R|$. Therefore $|B\cup R|\geq \frac{1}{2}|V(G)|$ as required.\end{proof}

\section{Discussion}
Theorem \ref{thm:higherpowers} answers the question of giving a lower bound on the number of edges that are gained by taking higher powers of a graph.  We obtain growth that is linear with $r$ -- just as in Theorem \ref{thm:cauchydavenport}.
\begin{itemize}

\item
The constant $\left \lceil \frac{1}{3}r \right \rceil$ in Theorem \ref{thm:higherpowers} cannot be improved to something of the form $\lambda r$ with $\lambda > \frac{1}{3}$.  To see and consider the following sequence of graphs $H_r(d)$ as $d$ tends to infinity:

Take disjoint sets of vertices $N_0,...,N_r$,  with $|N_i|=d-1$ if $i\equiv 0 \pmod 3$ and $|N_i| = 2$ otherwise.  Add all the edges within each set and also between neighboring ones.  So if $u \in N_i$, $v\in N_j$, then $uv$ is an edge whenever $|i-j|\leq 1$ (see Figure \ref{fig:img1}).

\begin{figure}[htp]
\centering
\includegraphics[scale=0.2,width=100mm]{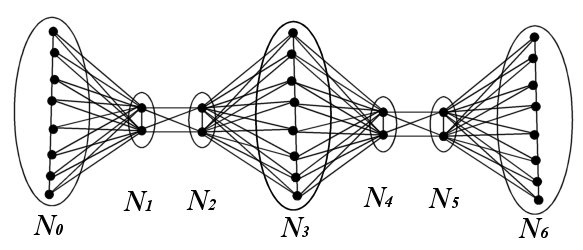}
\caption[My Image] {The graph $H_6(9)$.}\label{fig:img1}
\end{figure}

The number of edges in $H_r(d)$ is at least the number of edges in the larger classes which is $\left\lceil \frac{1}{3}(r+1) \right\rceil {d-1 \choose 2} $.

The $r$th power $H_r(d)^r$ has less than ${|V(G)|\choose 2}$ edges which is less than ${\lceil \frac{1}{3}(r+1) \rceil(d+3) \choose 2}$.  Therefore,
\begin{equation*}
\limsup_{d \to \infty} \frac{\edges(H_r(d)^r)}{\edges(H_r(d))} \leq \lim_{d\to \infty}  \frac{{\lceil \frac{1}{3}(r+1) \rceil(d+3) \choose 2}}{\left\lceil \frac{1}{3}(r+1) \right\rceil {d-1 \choose 2} } = \left \lceil \frac{1}{3}(r+1) \right \rceil. \label{eq:limit} %\tag{4}
\end{equation*}

The graphs $H_r(d)$ are not regular, but if $r \not\equiv 2 \pmod 3$, it
is possible to remove a small (less than $|V(G)|$) number of edges
from the graphs and make them $d$-regular without losing connectedness
(any cycle passing through all the vertices in $N_1 \cup ... \cup N_{r-1}$ would work).  Call
these new graphs $\hat H_r(d)$.  By the same argument as before we have
$$\limsup_{d \to \infty} \frac{\edges(\hat H_r(d)^r)}{\edges( \hat H_r(d))}
\leq \left \lceil \frac{1}{3}(r+1) \right \rceil.$$

If $r \equiv 2 \pmod 3$, a similar trick can be performed, but we'd
need to start with $|N_i|=d-1$ if $i\equiv 1 \pmod 3$ and $|N_i| = 2$
otherwise.

So the factor of $\frac{1}{3}$ cannot be improved for regular graphs.  All these examples are inspired by one given in
\cite{Hegarty} to show that for any $\epsilon$ there are regular
graphs $G$ with $e(G^2) < (1+\epsilon)e(G)$.
\item
Despite the above example, there is certainly room for further improvement in Theorems \ref{thm:higherpowers} and \ref{thm:improvedhegarty}.  In particular, Theorem \ref{thm:improvedhegarty} doesn't seem tight in any way.  The graphs $\hat H_r(d)$  seem to give essentially the slowest possible growth for \emph{all} powers of regular graphs.  Considering the graphs $H_3(d)$ leads to the conjecture of $$e(G^3)\geq 2e(G),$$   for $G$ regular, connected, and  $\diam(G)\geq 3$.

A shortcoming of Theorem \ref{thm:higherpowers} is that it only gives a good bound if the diameter of $G$ is close to $r$.  When this is not the case, the number of edges in $G^r$ seems to grow faster.  It would be interesting to obtain a good lower bound on $e(G^r)$ involving both $r$ and $\diam(G)$.
%Theorem \ref{thm:higherpowers} no longer gives a good bound if we insist on $\diam(G)\gg r$.  In that case the same method can be used to show that $$\edges(G^{r}) \geq \left( 2\left\lceil \frac{r}{3} \right\rceil -1 \right)\edges(G).$$
%This comes from (\ref{eq:vsum}) noting that in a regular graph with sufficiently large diameter, almost all the vertices lie distance greater than $r$ from a fixed vertex $v$.

\item
All the questions from this paper and \cite{Hegarty} could be asked for directed graphs.  In particular one can define directed Cayley graphs for a set $A \subseteq \mathbb{Z}_p$ by letting $xy$ be a directed edge whenever $x-y \in A$.  Then the Cauchy-Davenport Theorem implies an identical version of Theorem \ref{thm:cauchydavenport} for directed Cayley graphs.
In this setting it is easy to show that there is growth even for the square of an out-regular oriented graph $D$ (a directed graph where for a pair of vertices $u$ and $v$, $uv$ and $vu$ are not both edges).  In particular, we have
\[
\edges(D^2)\geq \frac{3}{2}\edges(D).
\tag{8}
\]
This occurs because every vertex $v$ has $|N_2 ^{out}(v)|\geq \frac{1}{2}|N_1 ^{out}(v)|$ in an out-regular oriented graph.  It's easy to see that this is best possible for such graphs.  One can construct out-regular oriented graphs with an arbitrarily large proportion of vertices $v$ satisfying $|N_2 ^{out}(v)| = \frac{1}{2}|N_1 ^{out}(v)|$.

However if we insist on \emph{both} in and out-degrees to be
constant, (8) no longer seems tight.  Such graphs are always
Eulerian.  In \cite{Seymour} there is a conjecture attributed to
Jackson and Seymour that if an oriented graph $D$ is Eulerian, then
$\edges(D^2)\geq 2 \edges(D)$ holds.  If this conjecture were proved,
it would be an actual generalization of the directed version of
Theorem \ref{thm:cauchydavenport}, as opposed to the mere analogues proved above.

\end{itemize}
\section{Acknowledgement}
I would like to thank my supervisors Jan van den Heuvel, and Jozef Skokan for much helpful advice and remarks.

\end{document}